\documentclass[11pt,a4paper]{amsart} 
\usepackage[centering,left=4.5cm,right=4.5cm]{geometry}
\usepackage{url}
\usepackage{mathrsfs}
\usepackage{enumitem}
\usepackage{color} 
\usepackage[colorlinks,linkcolor=blue,anchorcolor=blue,citecolor=blue,backref=page]{hyperref}
\usepackage{graphics,epsfig}
\usepackage{graphicx}
\usepackage{float}
\usepackage{epstopdf}
\usepackage[utf8]{inputenc}
\usepackage{hyperref} 
\usepackage{calc, amscd}
\usepackage[T1]{fontenc} 
\usepackage{pdfpages}
\usepackage{mathtools} 
\usepackage{scalerel,stackengine}
\usepackage{caption}
\usepackage{mathptmx}
\usepackage[per-mode=symbol]{siunitx}
\usepackage{bm}
\usepackage{amsmath}
\hypersetup{breaklinks=true}
\usepackage{amsthm,mathrsfs,multirow,xcolor,lscape,longtable,pbox,lipsum}
\usepackage{amsfonts,amscd,amssymb}
\usepackage[thinlines]{easytable}
\usepackage{microtype}
\usepackage{todonotes}
\usepackage{phaistos}
\usepackage{tabularx}
\newtheorem{theorem}{Theorem}

\newtheorem{lemma}[theorem]{Lemma}
\newtheorem{corollary}[theorem]{Corollary}

\newtheorem{remark}{Remark} 
\newtheorem{question}{Question}

\newtheorem*{theorem*}{Theorem}
\newtheorem*{question*}{Question}

\usepackage{mathtools}

\newcommand{\tmt}[4]{\left({#1\atop #3}{#2\atop #4}\right)}

\newcommand{\mr}[1]{\mathrm{#1}}

\begin{document}
\date{\today}
\author{Subham Bhakta}
\address{Mathematisches Institut, Georg-August-Universit\"at G\"ottingen, Germany.} \email{subham.bhakta@mathematik.uni-goettingen.de}
%\address{Indian Institute of Science Education and Research, Thiruvananthapuram, India} 
%\email{muneeswaran20@iisertvm.ac.in}
%\address{Indian Institute of Science Education and Research, Thiruvananthapuram, India} 
%\email{srilakshmi@iisertvm.ac.in}
\subjclass[2010]{Primary 11F30, 11L07; Secondary 11P05, 11B37, 11F80}
\keywords{Linear recurrence sequence, exponential sums, modular forms}

\newcommand\G{\mathbb{G}}
\newcommand\T{\mathbb{T}}
\newcommand\sO{\mathcal{O}}
\newcommand\sE{{\mathcal{E}}}
\newcommand\tE{{\mathbb{E}}}
\newcommand\sF{{\mathcal{F}}}
\newcommand\sG{{\mathcal{G}}}
\newcommand\sH{{\mathcal{H}}}
\newcommand\sN{{\mathcal{N}}}
\newcommand\GL{{\mathrm{GL}}}
\newcommand\HH{{\mathrm H}}
\newcommand\mM{{\mathrm M}}
\newcommand\fS{\mathfrak{S}}
\newcommand\fP{\mathfrak{P}}
\newcommand\fQ{\mathfrak{Q}}
\newcommand\Qbar{{\bar{\Q}}}
\newcommand\sQ{{\mathcal{Q}}}
\newcommand\sP{{\mathbb{P}}}
\newcommand{\Q}{\mathbb{Q}}
\newcommand{\tH}{\mathbb{H}}
\newcommand{\Z}{\mathbb{Z}}
\newcommand{\R}{\mathbb{R}}

\newcommand\Gal{{\mathrm {Gal}}}
\newcommand\SL{{\mathrm {SL}}}
\newcommand\Hom{{\mathrm {Hom}}}

\newtheorem{thm}{Theorem}[section]
\newtheorem{ack}[thm]{Acknowledgement}
\newtheorem{cor}[thm]{Corollary}
\newtheorem{conj}[thm]{Conjecture}
\newtheorem{prop}[thm]{Proposition}

\theoremstyle{definition}

\newtheorem{claim}[thm]{Claim}

\theoremstyle{remark}

\newtheorem*{fact}{Fact}
\newcommand{\SK}[1]{{\color{red} \sf $\heartsuit\heartsuit$ Srilakshmi: [#1]}}
\newcommand{\RM}[1]{{\color{blue} \sf $\heartsuit\heartsuit$ Muneeswaran: [#1]}}
\newcommand{\SB}[1]{{\color{purple} \sf $\heartsuit\heartsuit$ Subham: [#1]}}

\title[Galois representations and composite moduli]{Galois representations and composite moduli}
\begin{abstract}
    It is known that for any elliptic curve $E/\mathbb{Q}$ and any integer $m$ co-prime to $30,$ the induced Galois representation $\rho_{E,m}: \text{Gal}(\overline{\mathbb{Q}}/\mathbb{Q}) \longrightarrow \text{GL}_{2}(\mathbb{Z}/m\mathbb{Z})$ is surjective if and only if $\rho_{E,\ell}$ is surjective for any prime $\ell|m.$ In this article, we shall discuss some generalizations, applications, and variants of this phenomenon. 
\end{abstract}
\maketitle 
\tableofcontents
\section{Introduction} \noindent Let $E/\mathbb{Q}$ be an elliptic curve. Serre introduced the following representation, $$\rho_{E,n} :\text{Gal} (\overline{\mathbb{Q}}/\mathbb{Q}) \longrightarrow \text{Aut}_{\mathbb{C}}(E[n])\cong \text{GL}_{2}\big(\mathbb{Z}/n\mathbb{Z}\big),$$
where $E[n]$ is the set of $n$-torsion points in $E(\mathbb{C}).$ Serre's open image theorem says that, if $E$ is without complex multiplication, then there exists a constant $c_{E}>0$ such that for any prime $\ell>c_{E},$ the associated representation $\rho_{E,\ell}$ is surjective. It is conjectured that $c_{E}$ is uniformly bounded. For the known bounds on $c_E,$ the reader may refer to \cite{Coj} and \cite{Zyw15}. When $E$ has complex multiplication, the surjectivity is not true for large primes; see page 12 in \cite{Camp}. In general, whether the elliptic curve $E$ is without complex multiplication or not, Serre showed that for any $m \in \mathbb{N}$ with $(m,30)=1$, $\rho_{E,m}$ is surjective if and only if $\rho_{E,\ell}$ is surjective for any prime $\ell\mid m$. For reference, the reader may look at \cite{Coj} and \cite{Serre68}. We shall call this phenomenon $\textit{local-global property of Galois representations}$.

We first consider a similar problem over arbitrary number fields. Let $K$ be a number field, and $E$ be an elliptic curve defined over $K.$ We now have an associated Galois representation defined by 
\[\rho_{E,n}: \text{Gal}(\overline{K}/K) \longrightarrow \text{Aut}_{\mathbb{C}}(E[n]) \cong \text{GL}_2(\mathbb{Z}/n\mathbb{Z}).\]
The absolute Galois group $\text{Gal}(\overline{K}/K)$ is contained in $\text{Gal}(\overline{\mathbb{Q}}/\mathbb{Q}),$ but since $E$ is \textit{not necessarily} defined over $\mathbb{Q},$ one can not possibly extend $\rho_{E,n}$ to $\text{Gal}(\overline{\mathbb{Q}}/\mathbb{Q}).$ Following Cojocaru’s arguments, and the ramification theory for cyclotomic fields, it is not hard to show that a similar result holds for any $m$ with $(m, 30d_K)=1,$ or $(\phi(m), D_k)=1,$ where $d_K$ and $D_k$ respectively are the discriminant and degree of $K$ over $\mathbb{Q}.$ We shall have a detailed discussion on this in Theorem~\ref{thm:nf version}.

Let us call a natural number $m, \textit{bad}$ if there exists a finite extension $K_m$ of $K,$ and an elliptic curves $E$ over $K_m$ such that $\text{im}(\rho_{E,m})\neq\text{GL}_2(\mathbb{Z}/m\mathbb{Z})$ but $\text{im}(\rho_{E,\ell})=\text{GL}_2(\mathbb{Z}/\ell\mathbb{Z})$ for any prime $\ell \mid m.$ Moreover, we call such a number field as $\textit{native}$ to $m$ and such an elliptic curve as $\textit{exceptional elliptic curve}$ for the pair $(m,K_m).$ This is how we measure the failure of local-global property for Galois representations.

For any arbitrary number $K,$ we arrange the elliptic curves over $K$ with respect to the usual height $h(E)=||(a,b)||,$ where we consider the usual norm in $\mathbb{R}\otimes \mathcal{O}_K^2.$ Here $E$ is in the Weierstrass form given by $E_{(a,b)}:y^2=x^3+ax+b$ with $a,b\in \mathcal{O}_K,$ the ring of integers of $K.$ Denote $S_K(x)=\{(a,b)\in\mathcal{O}^2_K\mid h(E_{(a,b)})\leq x\}.$ It can be shown that $|S_K(x)|=c_Kx^{2D_K},$ for some constant $c_K>0.$ 

Let us now state the main result of this article.

\begin{theorem} \label{thm:nf version}
Let $K$ be a number field with discriminant $d_K$ and degree $D_k$ over $\mathbb{Q}.$ Consider $E/K$ to be an elliptic curve and $m$ be any natural number co-prime to $30.$ 
\begin{enumerate}
\item[(a)]  The induced Galois representation $\rho_{E,m}$ is surjective if and only if $\rho_{E,\ell}$ is surjective for any prime $\ell|m,$ provided that $K$ contains no proper abelian extension of $\mathbb{Q},$ or if $m$ is co-prime to the discriminant $D_K.$\\
\item [(b)] Any integer $m$ co-prime to $30,$ that is not square-free, is bad. Moreover, for any number field, $K_m$ native to $m,$ almost all the elliptic curves over $K_m$ are exceptional. 
\end{enumerate}
\end{theorem}

Following the same argument, one can show that for any arbitrary number field $K, \text{im}(\rho_{E,m})$ contains $\text{SL}_2(\mathbb{Z}/m \mathbb{Z})$, if and only if $\text{im}(\rho_{E,\ell})$ contains $\mathrm{SL}_2(\mathbb{Z}/m \mathbb{Z})$ for any prime $\ell|m,$ assuming that $(m,30)=1.$ In particular, it follows from Proposition 5.2 of \cite{Zyw10} that
\begin{corollary}
Let $K$ be any arbitrary number field and $m$ be any given integer co-prime to $30.$ Then for almost all the elliptic curves $E$ over $K,~\mathrm{im}(\rho_{E,m})$ contains $\mathrm{SL}_2(\mathbb{Z}/m\mathbb{Z}).$
\end{corollary}

Following the notations introduced by Grant in \cite{Gra}, we denote $\mathcal{E}_m(X)$ to be the number of elliptic curves (up-to isomorphism) over $\mathbb{Q}$ of height at most $X$, such that the corresponding representation $\rho_{E, m}$ is not surjective. Let $m$ be a natural number with $(m, 210)=1,$ then it follows from Cojocaru’s and Grant’s result that, \[\mathcal{E}_m(X)=\sum_{\ell|m} \mathcal{E}_{\ell}(X) \leq \omega(m) X^{1+\varepsilon},\]
where $\omega(m)$ denotes the number of distinct prime factors of $m.$ 

Zywina later considered a related problem over arbitrary number fields in \cite{Zyw10}. The interesting part comes when we want to make Zywina's bounds effective. For instance, if we look at Proposition 5.7 in \cite{Zyw10}, it is clear that the absolute constant is given by the number of conjugacy classes of $\text{GL}_2(\mathbb{Z}/m\mathbb{Z})$ with determinant $1$. Therefore, the absolute constant is about $\frac{m^3}{\phi(m)}.$ By Theorem~\ref{thm:nf version}, the constant is of order $\sum_{\ell \mid m} \frac{\ell^3}{\phi(\ell)},$ for any $m \in \mathbb{N}$ with $(m, 30d_K)=1.$ Of course, this constant is better when $m$ is not square-free, and the prime factors are with sufficiently large valuations. 

Serre introduced a representation associated with the pair of elliptic curves (analogously for the arbitrary tuple as well) as, 

$$\rho_{E_1\times E_2,n}(\sigma)=\tmt{\rho_{E_1,n}(\sigma)}{0}{0}{\rho_{E_2,n}(\sigma)}~
   \mathrm{for}~\mathrm{all}~\sigma \in \text{Gal} (\overline{\mathbb{Q}}/\mathbb{Q}).$$
 In this case, Serre showed an analog of his open image theorem. To be more precise, Serre showed that $\text{im}(\rho_{E_1\times E_2, \ell})=\Delta(\ell)$
for all but finitely many primes $\ell$, where the diagonal sub-group $\Delta(\ell)$ is given by,
$$\Delta(\ell) = \left \{ \tmt{g_1}{0}{0}{g_2}\mid  \text{det} (g_{1})=\text{det} (g_{2}), \big(g_{1},g_{2}\big) \in \text{GL}_{2}(\mathbb{Z}/\ell\mathbb{Z}) \times \text{GL}_{2}(\mathbb{Z}/\ell\mathbb{Z}) \right \}.$$ 
Jones in \cite{Jones13} considered this topic and proved an asymptotic estimate analogous to Grant's main result in \cite{Gra}. Grant's work was based on counting rational points on certain modular curves. Jones's approach was by studying the distribution of Frobenius symbols using the multi-dimensional version of Gallagher's large sieve. 

\begin{theorem} \label{thm:product} 

Let $K$ be a number field with discriminant $d_K$ and degree $D_k$ over $\mathbb{Q}.$ Consider $(E_1, E_2)$ be any pair of elliptic curves over $K$, and $m$ be a natural number, and $m$ be any natural number co-prime to $30.$ The induced Galois representation $\rho_{E_1\times E_2,m}$ has image $\Delta(m)$ if and only if, $\rho_{E,\ell}$ has image $\Delta(\ell)$ for any prime $\ell|m,$ provided that $K$ contains no proper abelian extension of $\mathbb{Q},$ or if $m$ is co-prime to the discriminant $D_K.$ 
%\end{enumerate}
\end{theorem}

In Section~\ref{sec:general}, we shall discuss a modular analog of the local-global phenomenon. We shall see that Theorem~\ref{thm:nf version} corresponds to newforms, and Theorem~\ref{thm:product} correspond to any arbitrary cusp form.

For polarized abelian varieties, it is believed that (see \cite{Ha}) an analog of Serre's open image theorem is true. If we trust this, we should have an analog of Theorem~\ref{thm:nf version}. Indeed, we have the following
\begin{theorem} \label{thm:abv version}. Let $K$ be any arbitrary number field, and $A/K$ be a polarized abelian variety of dim $g$. For any integer, $m\in \mathbb{N},$ the induced Galois representation $\rho_{A,m}$ is surjective if and only if $\rho_{A,\ell}$ is surjective for any prime $\ell|m$, provided that one of the conditions $(a)$ or $(b)$ in Theorem~\ref{thm:nf version} holds.
%\vspace{0.3cm}
%\item[(b)] Suppose that given any finite set of prime $S,$ there exists a polarized abelian variety $A$ over $K$ of dimension $g$ for which $\mr{im}(\rho_{A,\ell})$ contains $\mr{Sp}_{2g}(\mathbb{Z}/\ell\mathbb{Z}),\forall \ell\in S.$ Then the set of potentially bad numbers have a positive density in the set of natural numbers. 
%\end{enumerate}
\end{theorem}

Another paper by Jones (see \cite{Jones15}) focused on the required axioms, under which $\mr{H}=\mr{G}$ can be obtained from certain local conditions, where $\mr{H} \subseteq \mr{G} \subseteq \mathrm{GL_n}(\widehat{\mathbb{Z}})$ be algebraic groups. Our article primarily considers the vertical version of such a result. In \cite{Coj} Cojocaru first raised and answered such a question with $\mr{G}=\text{GL}_2(\mathbb{Z}/n\mathbb{Z})$ and $\mr{H}=\text{im}(\rho_{E})$ for an elliptic curve $E$ (better to say without complex multiplication) over $\mathbb{Q}.$ We first generalize the same over number fields. Jones applied his main result to the cases of the pair of elliptic curves and polarized abelian varieties. Taking inspiration from the same, we discuss a further generalization in section~\ref{sec:general} for a general class of algebraic groups. More precisely, our last section is about a discussion around the following question. First of all, let us denote $\pi_{m,\ell^r}$ to be the natural projection from $\mathbb{Z}/m\mathbb{Z}$ to $\mathbb{Z}/\ell^r\mathbb{Z},$ when $\ell^r \mid m.$ In the similar scenario, we denote $\pi_{m/\ell^r}$ to be the natural projection from $\mathbb{Z}/m\mathbb{Z}$ to $\mathbb{Z}/\frac{m}{\ell^r}\mathbb{Z}.$ For any integer $m$ and prime $\ell$ we also denote $\mathrm{pr}_{\ell}=\pi_{\ell^r,\ell}\circ \text{ker}\circ \pi_{m/\ell^r},$ where $\ell^r||m.$

\begin{question} \label{qn:general} Let $A$ be an alegebraic group, and $m$ be an integer. Take $G$ be a subgroup of $A(\mathbb{Z}/m\mathbb{Z}).$ Then is it true that
 \[\mr{G}=A(\mathbb{Z}/m\mathbb{Z})~\mr{if~and~only~if}~ \mr{pr}_{\ell}(G)=A(\mathbb{Z}/\ell\mathbb{Z})~\mr{for~any~prime}~\ell\mid m ?\]
\end{question}
We already have some examples of $A$ and $G$ for which we have a positive answer to the question for any integer $m$ co-prime to $30.$ For example, it follows from Theorem~\ref{thm:abv version} that we can take $A=\mathrm{GSp}_{2g}$ and $G$ to be an image of the Galois representation associated with some polarized abelian variety, 
\[\rho_{A,m}: \mathrm{Gal}(\overline{\mathbb{Q}}/\mathbb{Q}) \longrightarrow \mathrm{GSp}_{2g}(\mathbb{Z}/m\mathbb{Z}).\]
In Section~\ref{sec:general} we shall discuss this problem for a larger family.

\subsection{Overview of the article} The article is organized as follows. In Section~\ref{sec:backgrounds} we discuss the main results from finite groups. Nothing original is guaranteed in this section. Rather, it gives the reader a quick introduction to the subject and provides a notational setup for the article. In Section~\ref{sec:proofs}, we first prove one of our main results Theorem~\ref{thm:nf version}. One of the crucial results needed to prove this result is Lemma~\ref{lem:Ser1}.

Moreover, we also see how to extend these ideas to study more general cases involving pairs of elliptic curves and abelian varieties. We record these results in Theorem~\ref{thm:product} and Theorem~\ref{thm:abv version}. In this process, Lemma~\ref{lem:alg grp} plays the important role, which generalizes Lemma~\ref{lem:Ser1}. Finally, in Section~\ref{sec:general} we quickly discuss the possible generalizations of the results from previous sections, including the modular analogs.

\subsection{Notations} We denote $\mathbb{R}$ to be the set of all real numbers, $\mathbb{Q}$ to be the of all rational numbers, $\mathbb{N}$ to be the of all natural numbers and $\mathbb{Z}$ to be the of all integers.

Let $f,g$ be two functions supported on a subset of $\mathbb{R}.$ We write $f \ll g$ for $|f|\leq c|g| $ where $c$ is a constant irrespective of the domains of $f$ and $g$, often $f=O(g)$ is written to denote the same. In this article, such a constant is absolute, unless otherwise specified. Moreover, we write $f=o(g)$ to say that $\lim_{x\to\infty} \frac{|f(x)|}{|g(x)|}=0.$

For any integer $m,$ we denote $\phi(m)$ to be the Euler's totient function, and $\omega(n)$ be the number of distinct prime factors of $m.$ 

\section{Background on finite groups}\label{sec:backgrounds} 
This section discusses the main group theoretic tools required to prove the main results of this article. Although there may be many resources available on this, it serves the reader a quick introduction to the subject.

To begin with, let $\mr{G}$ be a finite group. Denote $\text{Occ}(\mr{G})$ to be the isomorphism classes of non-abelian simple groups coming as the quotient of composite factors of $\mr{G}$. First, we have the following immediate facts.

\begin{lemma}\label{Ser Occ} 
For any finite group $\mr{G}$,
\begin{enumerate}
    \item $\mathrm{Occ}(\mr{G})= \phi$ if and only if $\mr{G}$ is solvable. 
\item If $\mr{H}$ is normal in $G$, then $\mathrm{Occ}(\mr{G})=\mathrm{Occ}(\mr{H}) \bigcup \mathrm{Occ}(\mr{G}/\mr{H})$
\end{enumerate}
\end{lemma}
As a consequence, we have 
\begin{lemma} \label{Ser Occ 2}
For any integer $m,$
$$\mathrm{Occ}(\mathrm{GL}_{2}(\mathbb{Z}/m\mathbb{Z}))= \bigcup_{\substack{\ell,\mr{prime}\\\ell \mid m}} \mathrm{Occ}(\mathrm{PSL}_{2}(\mathbb{Z}/\ell\mathbb{Z})).$$
\end{lemma}

\subsection{A criteria for containing the simple linear groups of degree $2$}
Let us start by recalling the crucial results from \cite{Coj}. The following lemma is originally due to Serre (\cite{Serre68}).
\begin{lemma} \label{lem:Ser1}Let $\ell$ be a prime and $\mr{H}$ be any subgroup subgroup of $\mathrm{SL}_2(\mathbb{Z}/\ell^r\mathbb{Z}),$ then we have the following.
$$\mathrm{pr}_{\ell}(\mr{H})=\mathrm{SL}_2(\mathbb{Z}/\ell\mathbb{Z})\implies \mr{H}=\mathrm{SL}_2(\mathbb{Z}/\ell^r\mathbb{Z}),$$
where $\text{pr}_{\ell}$ denotes the canonical projection of $\mathrm{SL}_2(\mathbb{Z}/\ell^r\mathbb{Z})\to \mathrm{SL}_2(\mathbb{Z}/\ell\mathbb{Z}).$
\end{lemma}
Later, we shall prove this generalization in Section~\ref{sec:proofs}. Now the following result helps in computing the Occ of the finite groups.

\begin{corollary} \label{Coj Occ}
Let $\ell>5$ be any prime, and $\mr{G}$ be any arbitrary subgroup of $\mathrm{GL}_{2}(\mathbb{Z}/\ell^{r}\mathbb{Z})$. Then the following holds.
\[\mathrm{PSL}_{2}(\mathbb{Z}/\ell\mathbb{Z}) \in \mathrm{Occ}(\mr{G})~\text{if and only if}~\mathrm{SL}_{2}(\mathbb{Z}/\ell^{r}\mathbb{Z}) \subseteq \mr{G}.\]

\end{corollary}
Combining all these results, we have the following.
\begin{corollary} \label{Coj Occ 2} Let $m$ be any integer co-prime to $30,$ and $G$ be any subgroup of $\mathrm{GL}_2(\mathbb{Z}/m\mathbb{Z}).$ Then the following holds.  
$$\mathrm{SL}_{2}(\mathbb{Z}/m\mathbb{Z}) \subseteq \mr{G} ~\mr{if~and~only~if}~\mathrm{PSL}_{2}(\mathbb{Z}/\ell\mathbb{Z}) \in \mathrm{Occ}(\mr{G}),$$ for every prime $\ell|m.$
\end{corollary}
%he next result is due to Goursat, which has been used extensively in the recent times for product of elliptic curves.
%\begin{lemma}[\textbf{Goursat Lemma}]\label{Gour}
%Let $G_0$ and $G_1$ be groups and $G \subseteq G_0 \times G_1$ be a subgroup satisfying $\pi_i(G)=G_i,~i \in \{0,1\},$
%where $\pi_i$ denotes the canonical projection onto the $i^{th}$ factor.  Let $N_i= \pi_i(G \cap \mathrm{ker}(\pi_{1-i})).$ Then there is an isomorphism of groups $\psi:G_0/N_0 \rightarrow G_1/N_1$ for which \[G=\Big\{(g_0,g_1) \in G_0 \times G_1: \psi(g_0N_0)=g_1N_1\Big\}.\]
%\end{lemma}
%\begin{proof} See Lemma 3.2 in \cite{Jon2}

%\end{proof}
\subsection{Generalizations to the Symplectic groups}
Let us start by recalling the standard definition of Symplectic groups. Consider the groups, $$\mathrm{GSp}_{2g}(\mathbb{Z}/m\mathbb{Z})= \Big\{\mr{g} \in \mathrm{GL}_{2g}(\mathbb{Z}/m\mathbb{Z}) \mid \mr{g} \begin{pmatrix}
0 & I  \\
-I & 0 \\

\end{pmatrix}\mr{g}^{t}=\begin{pmatrix}
0 & \lambda  \\
-\lambda & 0 \\

\end{pmatrix}, \lambda \in (\mathbb{Z}/m\mathbb{Z})^{*} \Big\}$$ and $$\mathrm{Sp}_{2g}(\mathbb{Z}/m\mathbb{Z})= \Big\{\mr{g} \in \mathrm{GL}_{2g}(\mathbb{Z}/m\mathbb{Z}) \mid \mr{g} \begin{pmatrix}
0 & I  \\
-I & 0 \\

\end{pmatrix}\mr{g}^{t}=\begin{pmatrix}
0 & I  \\
-I & 0 \\

\end{pmatrix}\Big\}.$$ 
 We then have the following fact when we want to compute Occ.
\begin{lemma} \label{Occ Sp}
For any integer $m$, we have the following.
\[\mathrm{Occ}\Big(\mathrm{GSp}_{2g}(\mathbb{Z}/m\mathbb{Z})\Big)=\bigcup_{\substack{\ell~\mr{prime}\\\ell\mid m}} \mathrm{Occ}\Big(\mathrm{PSp}_{2g}(\mathbb{Z}/\ell\mathbb{Z})\Big).\]

\end{lemma}

\begin{proof} Since $\text{GSp}_{2g}(\mathbb{Z}/m\mathbb{Z})/\text{Sp}_{2g}(\mathbb{Z}/m\mathbb{Z})$ is abelian, 
\[\text{Occ}\Big(\mathrm{GSp}_{2g}(\mathbb{Z}/m\mathbb{Z})\Big)=\text{Occ}\Big(\mathrm{Sp}_{2g}(\mathbb{Z}/m\mathbb{Z})\Big)=\bigcup_{\ell^r||m}\text{Occ}\Big(\mathrm{Sp}_{2g}(\mathbb{Z}/\ell^r\mathbb{Z})\Big).\]
The last equality above follows from $(2)$ of Lemma~\ref{Ser Occ}. On the other hand,
\[\mathrm{Occ}\Big(\mathrm{Sp}_{2g}(\mathbb{Z}/\ell^r\mathbb{Z})\Big)=\mathrm{Occ}\Big(\mathrm{Sp}_{2g}(\mathbb{Z}/\ell\mathbb{Z})\Big)=\mathrm{Occ}\Big(\mathrm{PSp}_{2g}(\mathbb{Z}/\ell\mathbb{Z})\Big)\]
since kernel of the projection map $\mathrm{Sp}_{2g}(\mathbb{Z}/\ell^r\mathbb{Z}) \rightarrow \mathrm{Sp}_{2g}(\mathbb{Z}/\ell\mathbb{Z}),$ is a $\ell$-group, hence solvable, and $\text{Occ}$ of the solvable part is empty.
\end{proof}
As an immediate consequence
\begin{corollary} \label{Occ Sp 2} For any prime $\ell>5$ not dividing $m,$
\[\mathrm{PSp}_{2g}(\mathbb{Z}/\ell\mathbb{Z}) \not\in \mathrm{Occ}\Big(\mathrm{GSp}_{2g}(\mathbb{Z}/m\mathbb{Z})\Big).\]
\end{corollary}

\begin{proof} It follows immediately from Lemma \ref{Occ Sp} and Lemma 6.1 of \cite{Jones15}. Basically, Sylow's theorem shows that $\mathrm{PSp}_{2g}(\mathbb{Z}/\ell\mathbb{Z}) \not\in \mathrm{Occ}\Big(\text{PSp}_{2g}(\mathbb{Z}/\ell'\mathbb{Z})\Big),$
for any primes $\ell \neq \ell'.$
\end{proof}
Throughout the article, let us denote $\mr{comm}(G)$ to be the commutator of $G$ for any finite group $G.$ We then have,
\begin{lemma} \label{lem:comm}For any $m \in \mathbb{N}$ and $g>1,$
\[\mr{comm}(\mathrm{Sp}_{2g}(\mathbb{Z}/m\mathbb{Z}))=\mathrm{Sp}_{2g}(\mathbb{Z}/m\mathbb{Z}).\]

\end{lemma} 

\begin{proof} Let us first note that $\mathrm{Sp}_{2g}(\mathbb{Z}/m\mathbb{Z})=\prod_{\ell^r||m} \mathrm{Sp}_{2g}(\mathbb{Z}/\ell^r\mathbb{Z}).$ It is enough to prove the statement when $m$ is the power of a prime. After using Lemma~\ref{lem:alg grp}, which we shall prove in the next section, it is enough to show that $\mr{comm}(\mathrm{Sp}_{2g}(\mathbb{Z}/\ell \mathbb{Z}))=\mathrm{Sp}_{2g}(\mathbb{Z}/\ell\mathbb{Z}).$ Since $\mathrm{PSp}_{2g}(\mathbb{Z}/\ell \mathbb{Z})$ is simple and non-abelian, for any prime $\ell$ and $g>1$, we have that $\mr{comm}(\mathrm{PSp}_{2g}(\mathbb{Z}/\ell \mathbb{Z}))=\mathrm{PSp}_{2g}(\mathbb{Z}/\ell \mathbb{Z}).$
In particular, we have $$\mr{comm}(\pm \mathrm{Sp}_{2g}(\mathbb{Z}/\ell\mathbb{Z}))=\mathrm{Sp}_{2g}(\mathbb{Z}/\ell\mathbb{Z}).$$ We now need to show if $\mr{H}$ is a subgroup of $\mathrm{Sp}_{2g}(\mathbb{Z}/\ell\mathbb{Z})$ such that $\pm \mr{H}$ is the full group, then $\mr{H}$ itself is the full group $\text{Sp}_{2g}(\mathbb{Z}/\ell \mathbb{Z})$. One can see that the symplectic form $J=\begin{pmatrix}
 & I  \\
-I &  \\
\end{pmatrix}$ is in $\text{Sp}_{2g}(\mathbb{Z}/\ell\mathbb{Z})$. In particular, at least one of $J$ and $-J$ is in $\mr{H},$ because $\pm H= \text{Sp}_{2g}(\mathbb{Z}/\ell\mathbb{Z}).$ In either cases, we have $J^2=-I\in H.$ This concludes the proof, since 
$\pm H=H=\text{Sp}_{2g}(\mathbb{Z}/\ell \mathbb{Z}).$
\end{proof} 
\begin{lemma} \label{Square}
Let $A \in \mathrm{M}_{n}(\mathbb{Z}/\ell\mathbb{Z})$ be an arbitrary matrix with $\mathrm{tr}(A)=0$. Then $A$ is the sum of four square-zero matrices.
\end{lemma}
\begin{proof} See Theorem 1.1 in \cite{Paz}.
\end{proof}

\section{Proof of the main statements}\label{sec:proofs} 

In this section, we shall prove the main results of this section. To begin with, let $E$ be an elliptic curve over any arbitrary number field $K$.

\subsection{Proof of Theorem~\ref{thm:nf version}}
\subsection{Proof of part (a)} Let $m$ be an integer co-prime to $30.$ If $\rho_{E,m}$ is surjective, then $\rho_{E,\ell}=\text{pr}_{m,\ell} \circ  \rho_{E,m}$ is surjective for any $\ell \mid m,$ where 
\[\mathrm{pr}_{m,\ell}: \mathrm{GL}_2(\mathbb{Z}/m\mathbb{Z}) \longrightarrow \mathrm{GL}_2(\mathbb{Z}/\ell\mathbb{Z}),\]
is the natural projection.

For the converse, it follows from the given hypothesis that $\text{im}(\rho_{E,\ell})=\text{GL}_{2}(\mathbb{Z}/\ell\mathbb{Z})$ is a quotient of $\text{im}(\rho_{E,m})$ for any prime $\ell|m$. In particular, 
$$\text{PSL}_{2}(\mathbb{Z}/\ell\mathbb{Z}) \in \text{Occ}(\mr{G})$$ for any prime $\ell|m$, where $\mr{G}=\text{im}(\rho_{E,m}).$ It follows from Corollary~\ref{Coj Occ 2} that, $\text{SL}_{2}(\mathbb{Z}/m\mathbb{Z})$ is contained in $\mr{G}$. We then have that $\text{SL}_{2}(\mathbb{Z}/m\mathbb{Z}) =\mr{G}',$ and in particular $[G:G'] \mid \phi(m).$ On the other hand, the Weil pairing gives $[K(\zeta_{m}):K] \mid [G:G'].$ It is now enough to ensure that $[K(\zeta_{m}):K]= \phi(m)$. Note that, $[K(\zeta_{m}):K]=[\mathbb{Q}(\zeta_m):K\cap \mathbb{Q}(\zeta_m)],$ and hence it is enough to ensure that $K\cap \mathbb{Q}(\zeta_m)=\mathbb{Q}.$ We shall see the imposed conditions on $m$ or $K$ gives us that privilege.\footnote{This is not true in general. For instance, one may consider $K=\mathbb{Q}(\sqrt{-15})$ and then we have $[K(\zeta_{15}):K]=4\neq \phi(15)$.} 

First of all since $K\cap \mathbb{Q}(\zeta_m)$ is an abelian extension of $\mathbb{Q}$ contained in $K,$ the first imposed condition on $K$ forces the intersection to be trivial. On the other hand, since $K\cap \mathbb{Q}(\zeta_m)$ is an extension of $\mathbb{Q}$ contained in both $K$ and $\mathbb{Q}(\zeta_m),$ it is evident that the condition $(\phi(m),D_K)=1$ implies $K\cap \mathbb{Q}(\zeta_m)=\mathbb{Q}.$ Moreover, the assumption $(m,d_{K})=1$ immediately implies that the discriminant of $K\cap \mathbb{Q}(\zeta_m)$ is only $1.$ In particular, in all the cases we have $K\cap \mathbb{Q}(\zeta_m)=\mathbb{Q},$ as desired. 

 \subsection{Proof of part (b)} Take any integer $m$ that is not square-free. Write $m=\prod_{i=1}^{d}\ell_i^{e_i}$ and without loss of generality let us assume that $e_1>1.$ Consider $F$ to be an extension of $\mathbb{Q}$ contained in $\mathbb{Q}(\zeta_{\ell_1^{e_1}})$ of degree $\ell_1.$ We can do that, because $e_1>1$ by the assumption. It is evident that, $F\cap \mathbb{Q}(\zeta_{\ell_1})=\mathbb{Q}.$ Now for any $i>1,$ we have $\mathbb{Q}(\zeta_{\ell_i})\cap \mathbb{Q}(\zeta_{\ell_1^{e_1}})=\mathbb{Q},$ and in particular, we have $F\cap \mathbb{Q}(\zeta_{\ell_1})=\mathbb{Q},~\forall 1\leq i\leq d.$ Hence, for any $1\leq i\leq d$ we have
\[[F(\zeta_{\ell_i}):F]=[\mathbb{Q}(\zeta_{\ell_i}): F \cap \mathbb{Q}(\zeta_{\ell_i})]=\ell_i-1,\]
where the last implication is true because $F\cap \mathbb{Q}(\zeta_{\ell_i})=\mathbb{Q}.$ Let us now denote $K_m=F,$ and show that the pair $(m,K_m)$ satisfies all the necessary conditions for $m$ to be a potentially bad number. First, we need to show that there exists at least one elliptic curve $E$ over $K_m,$ for which 
\[\text{im}(\rho_{E,m})\neq \text{GL}_2(\mathbb{Z}/m\mathbb{Z})~\mr{but}~\text{im}(\rho_{E,\ell_i})=\text{GL}_2(\mathbb{Z}/\ell_i\mathbb{Z}),~\forall~1\leq i\leq d.\]
We know from \cite{Zyw10} that, there exists at least one elliptic curve $E/K_m$ for which $\text{im}(\rho_{E,\ell_i})\supset \text{SL}_2(\mathbb{Z}/\ell_i\mathbb{Z}),~\forall~1\leq i\leq d.$ In fact, this holds for almost all elliptic curves over $K_m$. From the construction, we know that $[K_m(\zeta_{\ell_i}):K_m]=\ell_i-1,~\forall 1\leq i\leq d.$ Now it follows from the argument of part (a) that, i.e., due to the Weil pairing that, $\text{im}(\rho_{E,\ell_i})=\text{GL}_2(\mathbb{Z}/\ell_i\mathbb{Z}),~\forall~1\leq i\leq d.$ 

On the other hand, for any elliptic curve $E/K_m,$ if the image of $\rho_{E,m}$ is $\text{GL}_2(\mathbb{Z}/m\mathbb{Z}),$ then we must have that $|K_m(\zeta_m):K_m|=\phi(m).$ This is because, it follows from Weil-pairing that $\zeta_m\in K_m(E[m])$ and $\sigma(\zeta_m)=\zeta_m^{\det(\rho_{E,m}(\sigma))},$ where $\zeta_m$ is the primitive $m^{\text{th}}$ root of unity. In particular, the fixed field of $\text{SL}_2(\mathbb{Z}/m\mathbb{Z})$ correspond to $K(\zeta_m).$ This shows that $[K_m(\zeta_m):K_m]=\phi(m).$ Instead, we have
\[[K_m(\zeta_m):K_m]=[\mathbb{Q}(\zeta_m):K_m \cap \mathbb{Q}(\zeta_m)]\leq [\mathbb{Q}(\zeta_m):K_m]< \phi(m),\]
since $K_m$ is a non-trivial extension of $\mathbb{Q}$ contained in $\mathbb{Q}(\zeta_m),$ a contradiction. 
\qed
\begin{remark}\rm
 In part $(b)$ of Theorem~\ref{thm:nf version}, we assume that $m$ is square-free. Note that this assumption is necessary. Otherwise, we immediately have the result because $$\text{GL}_2(\mathbb{Z}/\mathbb{Z})=\prod_{\substack{\ell,~\mathrm{prime}\\\ell \mid m}}\text{GL}_2(\mathbb{Z}/\mathbb{Z}).$$
\end{remark}

Let us now discuss some interesting consequences of Theorem~\ref{thm:nf version}. If one wants to make Proposition 5.7 in \cite{Zyw10} effective, one can see the effective constant is given by $\frac{m^3}{\phi(m)}.$ First let us recall the definition of the set $B_{K,m}(x)$ from \cite{Zyw10}.
\begin{corollary} For any $m \in \mathbb{N}$ with $(m, 30d_K)=1,$ the effective constant is given by
\[\sum_{\ell\mid m } \frac{\ell^3}{\phi(\ell)}.\]

\end{corollary}
\begin{proof} It follows from the proof of Proposition 5.7 in \cite{Zyw10}, that $|B_{K,\ell}(x)| \leq \frac {\ell^3}{\phi(\ell)} \frac{\log x}{x^{{\frac{[K:\mathbb{Q}]}{2}}}}.$
It is now enough to show that
\[B_{K,m}(x)\subseteq \bigcup_{\ell \mid m} B_{K,\ell}(X).\]
It follows from Theorem~\ref{thm:nf version} that $\bigcap_{\ell \mid m}B_{K,\ell}(x)^{c}\supseteq B_{K,m}(x)^c,$ and this completes the proof.
\end{proof}
Theorem~\ref{thm:nf version}~motivates us to study many generalizations. The immediate ones that come to our mind are Theorem~\ref{thm:product}~and~\ref{thm:abv version}, which we will discuss now. 
\subsection{Proof of Theorem~\ref{thm:product}}
\subsection{Proof of part (a)}

One direction is obvious. For the other direction, by Theorem~\ref{thm:nf version} we have that
\[\text{im}(\rho_{E_{1},m})=\text{im}(\rho_{E_{2},m})=\text{GL}_{2}(\mathbb{Z}/m\mathbb{Z}).\]
Let $G$ be $\text{im} (\rho_{E_{1}\times E_{2},m}) \subseteq \Delta(m)$. It follows from the given condition that $\Delta(\ell) $ is a quotient of $G$, for any prime $\ell \mid m.$ Denote $S\Delta(\ell)$ to be the set of elements in $\Delta(\ell)$ whose each block has determinant $1,$ and $G^{(\ell)}$ be $\ker \circ \pi_{m/\ell^{r}}(G) \subseteq \Delta(\ell^{r})$, where $r$ is the maximum power of $\ell$ dividing $m$, and $\pi_{m/\ell^r}$ be the natural projection $\Delta(m)\to \Delta(m/\ell^r).$

Moreover, we consider $G'=\text{pr}_{\ell}(G^{(\ell)}) \subseteq \Delta(\ell)$, and set 
$$G'_{1}= \Big\{g\in \text{GL}_2(\mathbb{F}_{\ell}) :  \begin{pmatrix}
I &   \\
 & g\\

\end{pmatrix} \in G'\Big\},~G'_{2}= \Big\{g\in \text{GL}_2(\mathbb{F}_{\ell}) :  \begin{pmatrix}
g &   \\
 & I\\

\end{pmatrix} \in G'\Big\}.$$ 
For the given condition, Since $\Delta(\ell)=G',$ for every prime $\ell \mid m.$ In particular, $\text{comm}(G')=\text{S}\Delta(\ell).$ It now follows from Lemma~\ref{lem:alg grp}, which we shall prove in the next section, that $\text{S}\Delta(\ell^r)\subseteq G^{(\ell)}.$

According to Lemma 3.3 in [6], $G \neq \Delta(\mathbb{Z}/m\mathbb{Z}) $ implies there exist a set $C_{1} \times C_{2} \subseteq \text{GL}_{2}(\mathbb{Z}/m\mathbb{Z}) \times\text{GL}_{2}(\mathbb{Z}/m\mathbb{Z})$, closed under conjugation such that $\text{det} (C_{1}) = \text{det} (C_{2}) = 1$ with $G \cap (C_{1} \times C_{2}) = \phi$. This implies that, there exists at least one element $(c_1,c_2)\not\in G$ for which $\det(c_1)=\det(c_2)=1.$ In particular, there exists a prime $\ell \mid m$ such that the component of $(c_1,c_2)$ associated to $\ell$ is not in $G^{(\ell)}$, which implies $S\Delta(\ell^r)\not\subseteq G^{(\ell)}.$ This contradicts the conclusion of the previous paragraph.
Let $E$ be an elliptic curve over an arbitrary number field $K$, and consider the following invariant
\[A(E)= 30 \prod_{\ell \in M_E}  \ell,\]
where $M_E$ is the set of primes $\ell \geq 7$ such that $\rho_{E,\ell}$ is not surjective. Now for a pair of elliptic curves $E_1\times E_2$ over $K$, let us consider 
\[A(E_1 \times E_2)=30 \prod_{\ell \in M_{E_1 \times E_2}} \ell,\]
and $M_{E_1 \times E_2}$ is the set of primes $\ell$ for which $\text{im}(\rho_{E_1\times E_2,\ell})\neq \Delta(\ell).$ It is clear that
\[\text{lcm}(A(E_1), A(E_2)) \mid  A(E_1 \times E_2).\] 
If they are not equal, then there exists a prime $\ell$ such that $\text{im}(\rho_{E_i,\ell})=\text{GL}_2(\mathbb{Z}/\ell\mathbb{Z})$ for $i \in \{1,2\}$, and $\text{im}(\rho_{E_1\times E_2, \ell}) \neq \Delta(\ell).$ Now by Lemma 5.1 of \cite{MW93}, $\rho_{E_1,\ell}$ and $\rho_{E_2,\ell}$ are conjugate up-to a quadratic character of $\text{Gal}(\overline{\mathbb{Q}}/\mathbb{Q})$. In fact, it follows from Proposition 1 in \cite{MW93} and Theorem~\ref{thm:product} the following.
\begin{corollary}
    Let $K$ be any number field satisfying one the conditions in part (a) of Theorem~\ref{thm:nf version}, and $E_1,~E_2/K$ be two elliptic curves without complex multiplication, which are not isogenous over $\overline{K}$. then we have the following equality 
    $$\mathrm{im}(\rho_{E_1\times E_2,m})=\mathrm{GL}_2(\Z/m\Z),$$
    for any integer $m$ co-prime to $A(E_1\times E_2)\ll \max\{h_1,h_2\}^{O(1)},$ where $h_1$ and $h_2$ respectively be the heights of $E_1$ and $E_2.$ 
\end{corollary}
\begin{proof}
    It follows from Theorem~\ref{thm:product} that $\mathrm{im}(\rho_{E_1\times E_2,m})=\mathrm{GL}_2(\Z/m\Z)$, for any integer $m$ co-prime to $A(E_1\times E_2).$ For an upper bound on $A(E_1\times E_2),$ the reader may look at Proposition 1 in \cite{MW93}.
\end{proof}
\begin{remark}
    Jones in \cite{Jones13} showed that almost all pairs of elliptic curves over $\overline{\mathbb{Q}}$ are pairwise non-isogenous.
\end{remark}

\subsection{Proof of Theorem~\ref{thm:abv version}}
We need the following generalization of Lemma \ref{lem:Ser1}.
\begin{lemma} \label{lem:alg grp}
Let $p$ be a prime, $r$ be any integr, and $\mr{G}$ be an algebraic subgroup of $\mathrm{SL}_{n}(\mathbb{F}_{p^r}).$ Suppose that, $\mr{H}$ is a subgroup of $\mr{G}$ with $\mathrm{pr}_p(\mr{H})  = \mathrm{pr}_p(\mr{G}),$ then we have $\mr{H}=\mr{G}.$ More generally, if $\mr{H} \subseteq \mr{G}(\mathbb{Q}_{p}),$ with $\mathrm{pr}_p(\mr{H}) = \mathrm{pr}_p(\mr{G})$, then $\mr{H} =\mr{G}(\mathbb{Q}_{p}).$ 
\end{lemma}
\begin{proof}
Let us start with proving the first part because the second part is just a rewording of the former. We proceed by induction, to show that for any $2\leq k\leq r,$
$$\text{pr}_{p^k}(\mr{H})=\text{pr}_{p^k}(\mr{G}).$$ 
For each such $k$, consider the exact sequence $$1 \longrightarrow \Big \{A \in \text{G}(\mathbb{F}_{p^{k}}): A \text{ mod } p^{k-1} \equiv I \Big \} \longrightarrow \text{G}(\mathbb{F}_{p^{k}}) \longrightarrow \text{G}(\mathbb{F}_{p^{k-1}}) \longrightarrow 1.$$ Proceeding by induction, it is enough to show that 
$$\mr{pr}_{p^k}(\mr{H}) \supseteq \Big \{A \in \text{G}(\mathbb{F}_{p^{k}}): A \text{ mod } p^{k-1} \equiv I \Big \}.$$ Take an element $A=I+p^{k-1}U \in G(\mathbb{F}_{p^k})$ with $U \in \text{M}_n(\mathbb{F}_{p})$, and we want to show $A\in \mr{pr}_{p^k}(H).$ Note that $\text{det} (A) =1$ in $\mathbb{F}_{p^k},$ implies $\text{tr}(U)=0$ in $\mathbb{F}_{p}$. By lemma \ref{Square}, we can write $U$ as a sum of four square-zero matrices. It follows from induction that there exists $h \in \mr{pr}_{p^{k}}(\mr{H})$ such that $$ h= I+p^{k-2}U+p^{k-1}V$$ 
for some $V\in M_n(\mathbb{F}_p),$ provided that $I+p^{k-2}U \in \mr{G}( \mathbb{F}_{p^{k-1}})$. Since $\mr{G}$ is an algebraic group, it is given by zeroes of some polynomials, say by $f_{1},f_2,\cdots, f_{t}$. The identity matrix $I_{n \times n}$ is an element of $G$, in particular, we have
\begin{equation}\label{eqn:cond}
f_i\big(I_{n \times n}\big)=0, \hspace{.1cm} \text{for all}  \hspace{0.15cm} 1 \leq i\leq t.
\end{equation}

Since $I+p^{r-1}U\in \mr{G}(\mathbb{F}_{p^r})$, we have $$ f_{s}\Big ( \big(\delta_{ij}+p^{k-1}U_{ij}\big)_{1\leq i,j \leq n^{2}} \Big ) =0 \text{ mod }p^{k}, \text{ for all } 1\leq s\leq t.$$
Now any such $f_s$ is sum of the monomials of form $\gamma (\prod g_{i,i})^{\alpha_i} (\prod_{i\neq j} g_{i,j}^{\beta_{i,j}}).$ Putting $g_{i,j}=\delta_{i,j}+p^{k-1} U_{i,j}$ in each such monomial, and considering modulo $p^k,$ we get the following sum
\begin{equation}\label{eqn:terms}
    \prod (1+\alpha_i U_{i,i} p^{k-1})(\prod_{i \neq j} (p^{k-1} U_{i,j})^{\beta_{i,j}}.
    \end{equation}
If all of the $\beta_{i,j}$'s are zero, then the above summation is congruent to 
$$\prod (1+\alpha_iU_{i,i}p^{k-1}) \equiv 1+p^{k-1}\sum \alpha_i U_{i,i} \pmod {p^k}$$
provided that $k\geq 2.$ Now if some $\beta_{i,j}$ is more than $1$, the equation at (\ref{eqn:terms}) is already congruent to $0$ modulo $p^k,$ provided $k \geq 2.$ Same goes if two of the $\beta_{i,i}$'s are exactly $1$. Now, if only one of the $\beta_{i,j}$ is $1$ and others are zero, the equation at (\ref{eqn:terms}) is congruent to $p^{k-1}U_{i,j}$ modulo $p^k$ provided $k \geq 2.$ Combining all these observations, one can see the equation at $(\ref{eqn:cond})$ is congruent to 
\[f_s\big(I_{n \times n}\big)+p^{k-1}(\text{some term only involving}~ \alpha_i, U_{i,j}'s),\forall 1\leq s\leq t\]
modulo $p^k$. A similar argument also shows $$ f_{s}\Big ( (\delta_{ij}+p^{k-2}U_{ij})_{1\leq i,j \leq n^{2}} \Big ) =0 \text{ mod }p^{k-1},~\forall~1\leq i\leq k,$$ 
provided that $k \geq 3.$ So $I+p^{k-2}U \in \text{G}(\mathbb{F}_{p^{k-1}}),$ as desired. By induction, we now have such a desired $h\in \mr{pr}_{p^k}(H)$. And hence,
\[A=h^p \in \mr{pr}_{p^k}(H), \forall~2\leq k\leq r.\]

\end{proof}
\begin{remark}\rm
    The lemma above is not necessarily true for algebraic groups not containing $\text{SL}_n$. For a justification, two examples are addressed in the next section. Roughly speaking, this is the case because $p \mid \alpha-1$ does not imply $p^r \mid \alpha-1,$ while on the other hand, $p^r \not|~\alpha-1$ implies $p \not|~\alpha-1.$
\end{remark}

We are now finished with all the preparations. Let us now prove the main result about Abelian varieties. 
\subsection{Proof of the Theorem~\ref{thm:abv version}}
\subsection{Proof of part (a)}
Again, one direction is immediate. For the other direction, let $\mr{G}$ be the image of $\rho_{A,m}$. It follows from the given condition that, $\text{GSp}_{2g}(\mathbb{Z}/\ell\mathbb{Z})$ is a quotient of $\mr{G}$ for any prime $\ell \mid m$. In particular, $ \text{PSp}_{2g}(\mathbb{Z}/\ell\mathbb{Z}) \in \text{Occ}(\mr{G}), \text{ for any prime } \ell|m,$  since $\text{PSp}_{2g}(\mathbb{Z}/\ell\mathbb{Z})$ is simple and non-abelian for any $g>1$. From lemma \ref{Occ Sp} and corollary \ref{Occ Sp 2}, we have that $$\mr{PSp}_{2g}(\mathbb{Z}/\ell\mathbb{Z}) \in \text{Occ}(\text{pr}_\ell(G^{(\ell)})),~\mr{where}~\mr{G}^{(\ell)}=\text{ker}_{m/\ell^r} \mr{G}.$$

For the case $g=1$ Cojocaru argued from this point by showing $\text{pr}_{\ell}(\mr{G}^{(\ell)})$ contains $\text{SL}_{2}(\mathbb{Z}/\ell\mathbb{Z}).$ Her argument was based on the fact that a subgroup of $\text{GL}_2(\mathbb{Z}/\ell\mathbb{Z})$ not containing $\text{SL}_2(\mathbb{Z}/\ell\mathbb{Z})$ is solvable. However, the same argument does not work for the general cases. For example $\text{Sp}_{2g-2}(\mathbb{Z}/\ell\mathbb{Z})$ is not a solvable subgroup of $\text{GSp}_{2g}(\mathbb{Z}/\ell\mathbb{Z}).$ But on the other hand, $\frac{\text{pr}_{\ell}(G^{(\ell)})} {\text{pr}_{\ell}(G^{(\ell)})\cap \text{Sp}_{2g}(\mathbb{Z}/\ell\mathbb{Z})} \subseteq \mathbb{F}_{\ell}^{*}$ is abelian. In particular, we have $\text{PSp}_{2g}(\mathbb{Z}/\ell\mathbb{Z}) \in \text{ Occ}(\text{pr}_{\ell}(G^{(\ell)}) \cap \text{Sp}_{2g}(\mathbb{Z}/\ell\mathbb{Z}))$. Comparing cardinality of them, we have $\text{pr}_{\ell}(G^{(\ell)}) \cap \text{Sp}_{2g}(\mathbb{Z}/\ell\mathbb{Z}) = \text{Sp}_{2g}(\mathbb{Z}/\ell\mathbb{Z})$, and hence $\text{pr}_{\ell}(G^{(\ell)}) \supseteq \text{Sp}_{2g}(\mathbb{Z}/\ell\mathbb{Z}).$ It now follows from Lemma~\ref{lem:comm} that, $\text{comm}(\text{pr}_{\ell}(G^{(\ell)}))=\text{Sp}_{2g}(\mathbb{Z}/\ell\mathbb{Z}).$

On the other hand, $\text{Sp}_{2g} \subset \text{SL}_{2g}$ is an algebraic group, and hence it follows from the previous lemma that $$\text{Sp}_{2g} (\mathbb{Z}/\ell^{r}\mathbb{Z})=\text{comm}(G^{(\ell)}).$$ This shows that, $\text{Sp}_{2g} (\mathbb{Z}/m\mathbb{Z}) \subseteq G \subseteq \text{GSp}_{2g}(\mathbb{Z}/m\mathbb{Z}).$ It again follows from Lemma~\ref{lem:comm} that $\text{comm}(G)=\text{Sp}_{2g}(\mathbb{Z}/m\mathbb{Z}).$

It is now enough to show that $\phi(m)\mid [G:G'].$ Since $A$ is equipped with a polarization, due to Weil pairing the primitive $m^{\text{th}}$ root of unity $\zeta_m$ is in $K(A[m]),$ and hence, $|\text{Gal}(K(A[m]):K))|$ divides $[G:G'].$ Now the argument is now essentially same as part (a) of Theorem~\ref{thm:nf version}, we only need to ensure that $K(\zeta_m)$ is an extension of $K$ of degree $\phi(m).$
\subsection{Proof of part (b)} Let $m$ be a number co-prime to $30$ which is not squarefree. We can then consider $S$ to be the set of prime factors of $m.$ By the assumption, there exists s simply polarized abelian variety over $K$ of dimension $g$ for which $\text{im}(\rho_{A,\ell})\supseteq \text{Sp}_{2g}(\mathbb{Z}/\ell),\forall \ell \in S.$ We can now consider the same $K_m$ as in the part (a) of Theorem~\ref{thm:nf version} to ensure that $m$ is indeed a \textit{potential bad} number in this settings.  
\qed
%\subsection{Congruence class of the Fourier coefficients}
%Given any integer $m,$ denote $K_m$ be the number field for which $\text{Gal}()$

\section{Further remarks}\label{sec:general} 
In this section, we shall discuss the analogous situations for modular forms and other algebraic groups. More precisely, we shall discuss the local-global property for the relevant Galois representations. 

\subsection{The modular analog}
Let $f(z)$ be any newform of weight $k$ and level $N$. It is known due to Deligne-Serre correspondence that, for any integer $m$ we have an associated Galois representation
\[\rho_{f,m}:\text{Gal}\left(\overline{\mathbb{Q}}/\mathbb{Q}\right) \longrightarrow \mathrm{GL}_2\left(\mathbb{Z}/m\mathbb{Z}\right),\]
such that $a(p)\pmod{m}\equiv \text{tr}\left(\rho_{f,m}(\text{Frob}_p)\right)$ for any prime $p \nmid Nm.$ When $f$ is without CM, it follows from \cite{Ribet85} that, there exists an integer $M_f$ such that for any integer $m$ co-prime to $M_f,$ the image of this representation is given by 
$$
\Delta_{k,m}=\left\{A \in \text{GL}_2\left(\mathbb{Z}/m\mathbb{Z}\right) \mid \det(A) \in ((\mathbb{Z}/m\mathbb{Z})^{*})^{k-1}\right\}.$$ 
Following the proof of Theorem~\ref{thm:nf version}, one could see that for any integer $m$ co-prime to $30,$ to guarantee $\text{im}(\rho_{f,m})$ contains $\text{SL}_2(\mathbb{Z}/m\mathbb{Z})$ if and only if, $\text{im}(\rho_{f,\ell})$ contains $\text{SL}_2(\mathbb{Z}/\ell\mathbb{Z})$ for any prime $\ell \mid m.$ In particular, one could perhaps get a smaller $M_f$, as long as we want the image to contain only $\text{SL}_2(\mathbb{Z}/m\mathbb{Z}).$ In this direction, we ask the following.
\begin{question}
    Is it true that $$\mathrm{im}(\rho_{f,m})=\Delta_{k,m}~\mathrm{if~and~only~if},~\mathrm{im}(\rho_{f,\ell})=\Delta_{k,\ell},$$ for any prime $\ell \mid m$?
\end{question}
It is not hard to notice that the answer to this question is yes, provided that $\zeta_m^{(k-1,\phi(m))}$ is in the field corresponding to $\text{ker}(\rho_{f,m}).$

Moreover, any cuspform $f(z)$ can be uniquely written as $c_1f_1+c_2f_2+\cdots c_rf_r,$ where $c_1,c_2,\cdots, c_r\in \overline{\mathbb{Q}}.$ One can attach a Galois representation $\rho_{f,m}:\mathrm{Gal}\left(\overline{\mathbb{Q}}/\mathbb{Q}\right) \to \mathrm{GL}_{2r}\left(\Z/m\Z\right)$ defined by the map
\begin{equation}\label{eqn:map}
\sigma \mapsto \begin{psmallmatrix}
    \rho_{f_{1,\ell}}(\sigma) & & \\
    & \rho_{f_{2, \ell}}(\sigma) & &\\
    & & \ddots &\\
    & & & \rho_{f_{r,\ell}}(\sigma)\\ 
  \end{psmallmatrix}.
  \end{equation}
In this case, the image is contained in $\Delta_{k_1,k_2,\cdots,k_r}(m),$ where $\Delta_{k_1,k_2,\cdots,k_r}(m)$ denotes the set of all block matrices of size $2\times 2$ in $\mathrm{GL}_{2}\left(\Z/m\Z\right)$ in which determinant of each block is a $k_i-1^{th}$ power of some element in the multiplicative group $(\mathbb{\Z}/{m\Z})^{*}.$ Arguing similarly as in the proof of Theorem~\ref{thm:product} one can see that, $\text{im}(\rho_{f,m})$ contains $\text{SL}_2(\mathbb{Z}/m\mathbb{Z})^r$ if and only if, $\text{im}(\rho_{f,\ell})$ contains $\text{SL}_2(\mathbb{Z}/\ell\mathbb{Z})^r$ for any prime $\ell \mid m.$ One can show that when $f_1,f_2,\cdots, f_r$ are not related by pairwise Dirichlet characters, $\text{im}(\rho_{f,m})$ contains $\text{SL}_2(\mathbb{Z}/m\mathbb{Z})^r$ for all but finitely many primes $\ell.$ Horeover, we again as this stronger question.

\begin{question}
   Is it true that $$\mathrm{im}(\rho_{f,m})=\Delta_{k_1,k_2,\cdots,k_r,m}~\mathrm{if~and~only~if},~\mathrm{im}(\rho_{f,\ell})=\Delta_{k_1,k_2,\cdots,k_r,\ell},$$ for any prime $\ell \mid m$? 
\end{question}

\subsection{On the local-global phenomenon} Motivating from the previous section let us discuss the most general situation in this direction. Let $A$ be an algebraic subgroup of $\mathrm{GL}_n$ and $\mr{G}$ be an arbitrary subgroup of $A(\mathbb{Z}/m\mathbb{Z})$ such that $\mathrm{pr}_{\ell}(G)= A(\mathbb{Z}/\ell\mathbb{Z})~\mathrm{for~any~prime}~\ell$ dividing $m.$ Then what can be said about the relation between $G$ and $A$? To be more precise, let us now recall Question~\ref{qn:general} that we raised in the introduction. In the previous section we already had an affirmative answer when $A=\mathrm{GSp}_{2g}$ and $G$ is image of the mod-$\ell$ Galois representation associated to some polarized abelian variety, and $m$ is an integer satisfying certain coprimality conditions. However this is not true in general. Here are a few examples to justify that.
\subsection{Examples} 

\subsubsection{} Let $A=\mathrm{GL}_2$ and $\mr{G}$ be a subgroup of $\mathrm{GL}_2(\mathbb{Z}/\ell^2\mathbb{Z})$ defined by

\[\mr{G}=\Big\{\begin{pmatrix}

a & b  \\
c & d\\
\end{pmatrix}\mid a,b,c,d \in \mathbb{Z}/\ell\mathbb{Z}, ad-bc \neq 0\Big\}.\]
\noindent It is clear that $\mr{G}$ is a subgroup of $\mathrm{GL}_2(\mathbb{Z}/\ell^2\mathbb{Z})$ with $$\text{pr}_{\ell}(G) =\mathrm{GL}_2(\mathbb{Z}/\ell\mathbb{Z}), \hspace{0.2cm}~\text{while}~\mr{G} \neq \mathrm{GL}_2(\mathbb{Z}/\ell^2\mathbb{Z}).$$

In the similar vein, we can also consider $A=\mathrm{GL}_n$ and $G$ be a subgroup of $A(\mathbb{Z}/\ell^2\mathbb{Z})$ given by a copy of $\mathrm{GL}_n(\mathbb{Z}/\ell\mathbb{Z})$. For a more non trivial example, one may consider $A$ to be the set of Borel elements in $\mathrm{GL}_2$ and $G$ to be the set of all Borel elements over $\mathbb{Z}/\ell\mathbb{Z}.$
\subsection{Positive answer in general terms}
We first give an affirmative answer to Question~\ref{qn:general} under the following conditions.
\begin{enumerate}
%\item  $\mr{comm}(A(\mathbb{Z}/m\mathbb{Z}))=\mr{comm}^{(2)}(A(\mathbb{Z}/m\mathbb{Z}))$ 
    \item  $\mr{comm}(A(\mathbb{Z}/\ell\mathbb{Z})) \subseteq \mathrm{pr}_{\ell}(G).$ 
    \item $|\mr{comm}(G)|=\frac{|A(\mathbb{Z}/m\mathbb{Z})|}{|\mr{comm}(A(\mathbb{Z}/m\mathbb{Z}))|}.$
\end{enumerate}
Since $G$ is a subgroup of $A(\mathbb{Z}/m\mathbb{Z}),$ it follows immediately from $(1)$ that
\[\text{pr}_{\ell}(\mr{comm}(G^{(\ell)})) = \mr{comm}(A(\mathbb{Z}/\ell\mathbb{Z})),\]
where $G^{(\ell)}=\text{ker}\circ \pi_{m/\ell^r}(G).$ It now follows from Lemma~\ref{lem:alg grp} that $\mr{comm}(G^{(\ell)})=\mr{comm}(A(\mathbb{Z}/\ell^r\mathbb{Z}),$ because $\mr{comm}(A(\mathbb{Z}/m\mathbb{Z}))$ is contained in $\text{SL}_{n}(\mathbb{Z}/m\mathbb{Z}).$ In particular, we have
\[\mr{A(\mathbb{Z}/m\mathbb{Z})} \subseteq \mr{G} \subseteq A(\mathbb{Z}/m\mathbb{Z}).\]
Condition $(1)$ and $(2)$ together now shows $\mr{G}=A(\mathbb{Z}/m\mathbb{Z}).$

Note that we are putting an extra local condition at $(2)$, and that is not exactly our question was originally about. From the arguments presented in the previous section, we see that $(2)$ is actually true for the cases related to abelian varieties. In general we do not necessarily have that luxury because, once we have $A(\mathbb{Z}/\ell\mathbb{Z})$ is a quotient of $\mr{G},$ we perhaps will not get a simple non-abelian quotient of $A(\mathbb{Z}/\ell\mathbb{Z})$. So to avoid $(2),$ we might try to put the condition that 
\begin{equation}\label{eqn:cond1}
\mr{Sp}_{2g}(\mathbb{Z}/m\mathbb{Z}) \subseteq A(\mathbb{Z}/m\mathbb{Z}),
\end{equation}
for some $g\leq n/2,$ say. This condition implies that $\text{PSp}_{2g}(\mathbb{Z}/\ell\mathbb{Z}) \in \text{Occ}\Big(G^{(\ell)}\Big).$ Now arguing similarly as the proof of Theorem 1.3, we have \[\text{Sp}_{2g}(\mathbb{Z}/m\mathbb{Z})\subseteq \mr{G} \subseteq A(\mathbb{Z}/m\mathbb{Z}).\]
In particular, 
\[|\mr{G}/\mr{comm}(\mr{G})| \leq |\text{Sp}_{2g}(\mathbb{Z}/m\mathbb{Z})/|\text{Sp}_{2k}(\mathbb{Z}/m\mathbb{Z})|=O(m^{6(g-k)}\phi(m)).\]
Summarizing the discussion, we see that we can remove $(2)$ at the cost of two more conditions. One of them is $(\ref{eqn:cond1}),$ and the other one is if we assume that $|\mr{G}/\mr{comm}(\mr{G})|$ is sufficiently larger than $m^{6(g-k)}\phi(m).$

%we have
%\begin{theorem} \label{final}Answer to Question~\ref{qn:general} is positive when
%\begin{enumerate} 
%    \item  $A'=A''$ with $A'(\mathbb{Z}/m\mathbb{Z})=A(\mathbb{Z}/m\mathbb{Z})'.$
%    \item $\mathrm{Sp}_{2k}(\mathbb{Z}/m\mathbb{Z}) \subseteq A(\mathbb{Z}/m\mathbb{Z}).$ 
%    \item $|(G \pmod m)/(G \pmod m)'| \gg m^{6(g-k)}\phi(m).$
%\end{enumerate}
%\end{theorem}
%\subsection{Product of abelian varieties} Analogously as in the case of product of elliptic curves, we may consider the same problem with product of abelian varieties. Let $A_1, A_2$ be two simply polarized abelian varieties, having dimension $g_1, g_2$ respectively. Let $G=\text{im}(\rho_{A_1 \times A_2,m}) \subseteq \mathrm{GSp}_{2g_1+2g_2}(\mathbb{Z}/m\mathbb{Z})$ such that,
%\[\rho_{A_1 \times A_2,p}=\mathrm{GSp}_{2g_1+2g_2}(\mathbb{Z}/p\mathbb{Z}), \hspace{0.15cm}\text{for all} \hspace{0.10cm} p\mid m.\]
%Following the notations as in Section 3, we can consider $A=\text{GSp}_{2g_1+2g_2}$ and note that 
%$$\mathrm{Sp}_{2g_1}(\mathbb{Z}/m\mathbb{Z}),\mathrm{Sp}_{2g_2}(\mathbb{Z}/m\mathbb{Z}) \subseteq A(\mathbb{Z}/m\mathbb{Z}).$$ It is evident that conditions $(1), (2)$ of Theorem $4.2.1$ holds with $k=g_1$ and $g_2.$ This means
%\[\mathrm{Sp}_{2g_1+2g_2}(\mathbb{Z}/m\mathbb{Z}) \subseteq G.\]
%So $[G:G'] \mid \phi(m),$ and arguing similarly as in the previous section, we get the desired result in this case as well.

\section*{Acknowledgement}\label{ack}
%\section{Acknowledgements}
I first thank the Mathematics Institute of Georg-August-Universit\"at G\"ottingen for providing a beautiful environment for study and research. I am also indebted to my supervisor, Harald Helfgott, and Jitendra Bajpai, for many helpful discussions. I always found them kind and helpful in all my needs. This work is financially supported by ERC Consolidator grant 648329 (GRANT).

\bibliographystyle{amsplain}
\bibliography{references}

\end{document}